\documentclass[a4paper,dvipdfmx]{amsart}

\usepackage[utf8]{inputenc}
\usepackage[english]{babel}

\usepackage[T1]{fontenc}
\usepackage{lmodern}  

\usepackage{my-math-symbol}
\usepackage{my-cleveref}
\usepackage{my-equation-numbering}
\usepackage{comment}
\usepackage{tikz}
\usetikzlibrary{cd}

\usepackage[non-sorted-cites,initials,alphabetic,nobysame]{amsrefs}

\AtBeginDocument{%
	\def\MR#1{}
}

\title[Nonnegativity of the CR Paneitz operator]{Nonnegativity of the CR Paneitz operator for embeddable CR manifolds}
\author{Yuya Takeuchi}
\address{Department of Mathematics \\ Graduate School of Science \\ Osaka University
	\\ 1-1 Machikaneyama-cho, Toyonaka, Osaka 560-0043, Japan}
\email{yu-takeuchi@cr.math.sci.osaka-u.ac.jp, yuya.takeuchi.math@gmail.com}



\thanks{This work was supported by JSPS Research Fellowship for Young Scientists
and JSPS KAKENHI Grant Number JP19J00063.}

\begin{document}

\begin{abstract}
	The nonnegativity of the CR Paneitz operator
	plays a crucial role in three-dimensional CR geometry.
	In this paper,
	we prove this nonnegativity for embeddable CR manifolds.
	This result gives an affirmative solution to the CR Yamabe problem
	for embeddable CR manifolds.
	We also show the existence of a contact form with zero CR $Q$-curvature
	and generalize the total $Q$-prime curvature to embeddable CR manifolds
	with no pseudo-Einstein contact forms.
	Furthermore,
	we discuss the logarithmic singularity of the Szeg\H{o} kernel.
\end{abstract}

\maketitle

\section{Introduction} \label{section:introduction}

For three-dimensional strictly pseudoconvex CR manifolds,
the nonnegativity of the CR Paneitz operator has been of great importance;
it is deeply connected to global embeddability~\cite{Chanillo-Chiu-Yang2012}
and the CR positive mass theorem~\cite{Cheng-Malchiodi-Yang2017}.
It is known that
there are closed non-embeddable CR manifolds of dimension three
with not nonnegative CR Paneitz operator;
see~\cite{Cheng-Malchiodi-Yang2017}*{Section 4.1} for example.
Hence it is natural to ask whether any closed embeddable CR three-manifold has nonnegative CR Paneitz operator.
In this direction,
Chanillo, Chiu, and Yang~\cite{Chanillo-Chiu-Yang2013} have proved that
the nonnegativity condition is stable under sufficiently small deformations
of strictly pseudoconvex real hypersurfaces in $\mathbb{C}^{2}$.
Moreover,
Case, Chanillo, and Yang~\cite{Case-Chanillo-Yang2016} have shown that
the nonnegativity of the CR Paneitz operator is preserved under real-analytic embeddable deformations
with uniform positivity of the Tanaka-Webster scalar curvature and the stability of the CR pluriharmonic functions.
In this paper,
we prove that the CR Paneitz operator is nonnegative for embeddable CR manifolds.
Indeed:

\begin{theorem}
\label{thm:nonnegativity-of-CR-Paneitz-operator}
	Let $(M, T^{1, 0} M)$ be a closed embeddable strictly pseudoconvex CR manifold of dimension three.
	Then the CR Paneitz operator is nonnegative,
	and its kernel consists of CR pluriharmonic functions.
\end{theorem}

Note that the kernel of the CR Paneitz operator has been studied by
Case, Chanillo, and Yang~\cites{Case-Chanillo-Yang2015,Case-Chanillo-Yang2016}
for deformations of embeddable CR three-manifolds.

Here we give an outline of the proof of \cref{thm:nonnegativity-of-CR-Paneitz-operator}
given in \cref{section:CR-Paneitz-operator};
our proof is inspired by that of the Siu-Sampson result for harmonic maps~\cites{Siu1980,Sampson1986}
in~\cite{Amoros-Burger-Corlette-Kotschick-Toledo1996}*{Chapter 6.1}.
Let $(M, T^{1, 0} M)$ be as in \cref{thm:nonnegativity-of-CR-Paneitz-operator}.
From \cites{Harvey-Lawson1975,Lempert1995},
it follows that $M$ bounds a strictly pseudoconvex domain $\Omega$ in a two-dimensional complex projective manifold $X$.
By using a K\"{a}hler form on $X$ and a defining function of $\Omega$,
we have an asymptotically complex hyperbolic K\"{a}hler form $\omega_{+}$ on $\Omega$;
see \cref{section:CR-Paneitz-operator} for details.
For $u \in C^{\infty}(M)$,
take a harmonic extension $\widetilde{u}$ of $u$ to $\Omega$.
Since $\omega_{+}$ is K\"{a}hler,
$\widetilde{u}$ satisfies $d d^{c} \widetilde{u} \wedge \omega_{+} = 0$,
where $d^{c} = (\sqrt{- 1} / 2)(\overline{\partial} - \partial)$.
Then we have $d d^{c} \widetilde{u} \wedge d d^{c} \widetilde{u} \leq 0$,
and so
\begin{equation}
	\int_{M} d^{c} \widetilde{u} \wedge d d^{c} \widetilde{u}
	= \int_{\Omega} d d^{c} \widetilde{u} \wedge d d^{c} \widetilde{u}
	\leq 0.
\end{equation}
On the other hand,
integration by parts on $M$ yields that
\begin{equation}
	\int_{M} d^{c} \widetilde{u} \wedge d d^{c} \widetilde{u}
	= - \int_{M} u (P_{\theta} u) \theta \wedge d \theta,
\end{equation}
where $\theta$ is a contact form on $M$
and $P_{\theta}$ is the CR Paneitz operator.
Thus we obtain the nonnegativity of $P_{\theta}$.
Moreover,
if the equality holds,
then $\widetilde{u}$ must be pluriharmonic,
and $u$ is CR pluriharmonic.

\cref{thm:nonnegativity-of-CR-Paneitz-operator} has various applications.
Let $(M, T^{1, 0} M)$ be a closed strictly pseudoconvex CR manifold of dimension three.
The \emph{CR Yamabe constant} $\mathcal{Y}(M, T^{1, 0} M)$ is defined by
\begin{equation}
	\mathcal{Y}(M, T^{1, 0} M)
	= \inf_{\theta} \Set{\int_{M} \Scal \cdot \theta \wedge d \theta | \int_{M} \theta \wedge d \theta = 1}.
\end{equation}
This value gives a CR invariant of $(M, T^{1, 0} M)$.
\cref{thm:nonnegativity-of-CR-Paneitz-operator}
and~\cite{Chanillo-Chiu-Yang2012}*{Theorem 1.4(b)} imply the following embeddability criterion.

\begin{corollary}
\label{cor:embeddability-criterion}
	Let $(M, T^{1, 0} M)$ be a closed strictly pseudoconvex CR manifold of dimension three
	with positive CR Yamabe constant.
	Then $(M, T^{1, 0} M)$ is embeddable
	if and only if the CR Paneitz operator is nonnegative.
\end{corollary}

It is interesting to ask whether we can remove the positivity of the CR Yamabe constant
from the assumption of \cref{cor:embeddability-criterion}.

It is known~\cite{Jerison-Lee1987} that
the CR Yamabe constant satisfies
\begin{equation}
	\mathcal{Y}(M, T^{1, 0} M)
	\leq \mathcal{Y}(S^{3}, T^{1, 0} S^{3}),
\end{equation}
where $(S^{3}, T^{1, 0} S^{3})$ is the standard CR sphere in $\mathbb{C}^{2}$;
note that $\mathcal{Y}(S^{3}, T^{1, 0} S^{3}) > 0$.
Combining \cref{thm:nonnegativity-of-CR-Paneitz-operator} with
a consequence of the CR positive mass theorem~\cite{Cheng-Malchiodi-Yang2017}*{Theorem 1.2},
we have the following

\begin{corollary}
\label{cor:CR-Yamabe-constant}
	Let $(M, T^{1, 0} M)$ be a closed embeddable three-dimensional strictly pseudoconvex CR manifold.
	Then $\mathcal{Y}(M, T^{1, 0} M) = \mathcal{Y}(S^{3}, T^{1, 0} S^{3})$
	if and only if $(M, T^{1, 0} M)$ is CR equivalent to $(S^{3}, T^{1, 0} S^{3})$.
\end{corollary}

The \emph{CR Yamabe problem}
is to find a contact form $\theta$ satisfying
\begin{equation} \label{eq:Yamabe-problem}
	\Scal
	= \mathcal{Y}(M, T^{1, 0} M),
	\qquad
	\int_{M} \theta \wedge d \theta
	= 1;
\end{equation}
such a contact form is called a \emph{CR Yamabe contact form}.
Recently,
Cheng, Malchiodi, and Yang~\cite{Cheng-Malchiodi-Yang2019-preprint} have found that
there exist no CR Yamabe contact forms on the Rossi sphere,
which is an example of non-embeddable CR manifolds.
Nonetheless,
the CR Yamabe problem is solvable affirmatively for embeddable three-dimensional CR manifolds.

\begin{theorem}
\label{thm:CR-Yamabe-problem}
	There exists a CR Yamabe contact form
	on any closed embeddable strictly pseudoconvex CR manifold of dimension three.
\end{theorem}

We will also apply \cref{thm:nonnegativity-of-CR-Paneitz-operator}
to the zero CR $Q$-curvature problem.
In a study of the Szeg\H{o} kernel,
Hirachi~\cite{Hirachi1993} has introduced a pseudo-Hermitian invariant
\begin{equation}
	Q_{\theta}
	= \frac{1}{6}(\Delta_{b} \Scal - 2 \Im \tensor{A}{_{1}_{1}_{,}^{1}^{1}}),
\end{equation}
which is now called the \emph{CR $Q$-curvature}.
Note that this name stems from the fact that this invariant coincides with
a CR analog of the $Q$-curvature;
see~\cite{Fefferman-Hirachi2003}.
Since $Q_\theta$ is expressed as a divergence,
its integral is identically zero.
Moreover,
the CR $Q$-curvature itself is identically zero for pseudo-Einstein contact forms;
that is, contact forms satisfying the equality
\begin{equation}
	\tensor{\Scal}{_{1}} - \sqrt{- 1} \tensor{A}{_{1}_{1}_{,}^{1}}
	= 0.
\end{equation}
The existence of such a contact form is equivalent to that of a Fefferman defining function if $M$ bounds a strictly pseudoconvex domain~\cites{Hirachi1993,Hislop-Perry-Tang2008};
in particular,
it always exists if the CR manifold is a real hypersurface in $\mathbb{C}^{2}$.
For these reasons,
it is natural to ask whether any CR manifold admits a contact form with zero CR $Q$-curvature.
Some authors have tackled this problem via the CR $Q$-curvature flow~%
\cites{Chang-Cheng-Chiu2007,Saotome-Chang2011,Chang-Kuo-Saotome2019-preprint}.
In this paper,
however,
we will take a more functional-analytic approach.

\begin{theorem}
\label{thm:zero-CR-Q-curvature}
	Let $(M, T^{1, 0} M)$ be a closed embeddable strictly pseudoconvex CR manifold of dimension three.
	There exists a contact form $\theta$ on $M$ with zero CR $Q$-curvature.
	Moreover,
	if $\hat{\theta}$ is also such a contact form,
	then $\hat{\theta} = e^{\Upsilon} \theta$ for a CR pluriharmonic function $\Upsilon$.
	In particular,
	in the case where $M$ admits a pseudo-Einstein contact form,
	$\theta$ is pseudo-Einstein if and only if $Q_{\theta} = 0$.
\end{theorem}

The existence of such a contact form gives
a generalization of the total $Q$-prime curvature to embeddable CR manifolds with no pseudo-Einstein contact forms.
Let $(M, T^{1, 0} M)$ be a closed strictly pseudoconvex CR manifold of dimension three
and $\theta$ a contact form on $M$.
The \emph{$Q$-prime curvature} $Q^{\prime}_{\theta}$ is defined by
\begin{equation}
	Q^{\prime}_{\theta}
	= \frac{1}{2} \Scal^{2} - 2 \abs{\tensor{A}{_{1}_{1}}}^{2} + \Delta_{b} \Scal.
\end{equation}
Case and Yang~\cite{Case-Yang2013} have proved that
the integral of the $Q$-prime curvature
\begin{equation}
	\overline{Q}^{\prime}(M, T^{1, 0} M)
	= \int_{M} Q^{\prime}_{\theta} \theta \wedge d \theta
\end{equation}
is independent of the choice of a contact form $\theta$ with zero CR $Q$-curvature,
and defines a CR invariant of $(M, T^{1, 0} M)$,
called the \emph{total $Q$-prime curvature}.
Hence,
if $(M, T^{1, 0} M)$ is embeddable,
then the total $Q$-prime curvature is always well-defined and gives a CR invariant.
The total $Q$-prime curvature and the CR Yamabe constant satisfy the following inequality,
which is a CR analog of~\cite{Gursky1999}*{Equation (1.4)}.

\begin{theorem}
\label{thm:Einstein-condition}
	Let $(M, T^{1, 0} M)$ be a closed embeddable three-dimensional strictly pseudoconvex CR manifold.
	Then
	\begin{equation}
	\label{eq:comparison-with-Q-prime-and-Yamabe}
		\overline{Q}^{\prime}(M, T^{1, 0} M)
		\leq \frac{1}{2} \mathcal{Y}(M, T^{1, 0} M)^{2}
	\end{equation}
	with equality if and only if $M$ admits a pseudo-Einstein contact form
	with vanishing Tanaka-Webster torsion.
	In particular if $c_{1}(T^{1, 0} M) \neq 0$ in $H^{2}(M, \mathbb{R})$,
	then the strict inequality holds.
\end{theorem}

Note that there exists a two-dimensional strictly pseudoconvex domain
whose boundary $M$ satisfies $c_{1}(T^{1, 0} M) \neq 0$ in $H^{2}(M, \mathbb{R})$~\cite{Etnyre-Ozbagci2008}*{Theorem 6.2}.
In particular,
the strict inequality in \cref{eq:comparison-with-Q-prime-and-Yamabe} holds for this $M$.

We also mention the case where the equality in \cref{eq:comparison-with-Q-prime-and-Yamabe} holds.
A CR manifold having a pseudo-Einstein contact form with vanishing Tanaka-Webster torsion
is known as a Sasakian $\eta$-Einstein manifold.
Such a manifold plays a crucial role in CR geometry;
see~\cites{Wang2015,Case-Gover2017-preprint,Takeuchi2018} for example.

\cref{thm:Einstein-condition} gives a generalization of~\cite{Case-Yang2013}*{Theorem 1.1}
for embeddable CR manifolds.

\begin{corollary}
\label{cor:characterization-of-sphere}
	Let $(M, T^{1, 0} M)$ be a closed embeddable strictly pseudoconvex CR manifold of dimension three
	with nonnegative CR Yamabe constant.
	Then
	\begin{equation}
		\overline{Q}^{\prime}(M, T^{1, 0} M)
		\leq \overline{Q}^{\prime}(S^{3}, T^{1, 0} S^{3})
	\end{equation}
	with equality if and only if $(M, T^{1, 0} M)$ is CR equivalent to $(S^{3}, T^{1, 0} S^{3})$.
\end{corollary}

We will also apply our results to the logarithmic singularity of the Szeg\H{o} kernel.
Let $\Omega$ be a strictly pseudoconvex domain
in a two-dimensional complex manifold
and $\theta$ a contact form on $\partial \Omega$.
Define the \emph{Hardy space} $\mathcal{H}_{\theta}(\Omega)$
as the set consisting of the boundary values of holomorphic functions on $\Omega$
that are $L^{2}$ with respect to the volume form $\theta \wedge d \theta$ on $\partial \Omega$.
The \emph{Szeg\H{o} kernel} $S_{\theta}(z, \overline{w})$
is the reproducing kernel of $\mathcal{H}_{\theta}(\Omega)$.
Fix a defining function $\rho$ of $\Omega$.
The boundary behavior of $S_{\theta}$ on the diagonal is given by
\begin{equation}
\label{eq:aymptotic-expansion-of-Szego-kernel}
	S_{\theta}(z, \overline{z})
	= \phi_{\theta} \rho^{- 2} + \psi_{\theta} \log (- \rho),
\end{equation}
where $\phi_{\theta}$ and $\psi_{\theta}$ are smooth functions on $\overline{\Omega}$.
Moreover,
the Tayler coefficients of $\psi_{\theta}$ at a given boundary point
are uniquely determined by the behavior of $\theta$ near the point~\cite{Boutet_de_Monvel-Sjostrand1976}.
Hirachi~\cite{Hirachi1993}*{Main Theorem} has proved that
\begin{equation}
\label{eq:boundary-value-of-logarithmic-singularity}
	\psi_{\theta} |_{\partial \Omega}
	= \frac{1}{4 \pi^{2}} Q_{\theta}.
\end{equation}
Hence $\psi_{\theta} = O(\rho)$ if and only if $\theta$ is of zero CR $Q$-curvature.
It is natural to ask when $\psi_{\theta}$ vanishes to higher-order on the boundary.

\begin{theorem}
\label{thm:vanishing-of-logarithmic-singularity-of-Szego-kernel}
	Let $\Omega$ be a two-dimensional strictly pseudoconvex domain
	whose boundary admits a pseudo-Einstein contact form.
	For a contact form $\theta$ on $\partial \Omega$,
	the function $\psi_{\theta}$ is of $O(\rho^{2})$ (resp.\ $O(\rho^{3})$)
	if and only if $\theta$ is pseudo-Einstein
	and $\partial \Omega$ is obstruction flat (resp.\ spherical).
\end{theorem}

Note that the above theorem has been obtained by Hirachi
in the case when $\Omega$ is a domain in $\mathbb{C}^{2}$
and $\partial \Omega$ has transverse symmetry~\cite{Hirachi1993}.
It is interesting to ask what happens when $\partial \Omega$ has no pseudo-Einstein contact forms.
We also mention a recent progress by Curry and Ebenfelt~\cites{Curry-Ebenfelt2018-preprint,Curry-Ebenfelt2019} on the question:
If $\psi_{\theta} = O(\rho^{2})$,
does $\psi_{\theta} = O(\rho^{3})$ hold?

This paper is organized as follows.
In \cref{section:preliminaries},
we recall some basic definitions and facts on CR manifolds.
In \cref{section:CR-pluriharmonic-functions-and-pseudo-Einstein-condition},
we discuss a CR analog of $d^{c}$,
which is closely related to CR pluriharmonic functions
and plays an important role in the proofs of our results in this paper.
\cref{section:CR-Paneitz-operator} is devoted to
proofs of \cref{thm:nonnegativity-of-CR-Paneitz-operator,thm:CR-Yamabe-problem}.
In \cref{section:CR-Q-curvature},
we tackle the existence problem of a contact form with zero CR $Q$-curvature.
\cref{section:logarithmic-singularity-of-the-Szego-kernel} deals with
the logarithmic singularity of the Szeg\H{o} kernel.

\section{Preliminaries} \label{section:preliminaries}

Let $M$ be a smooth three-dimensional manifold without boundary.
A \emph{CR structure} is a complex one-dimensional subbundle $T^{1,0} M$
of the complexified tangent bundle $TM \otimes \mathbb{C}$ such that
\begin{equation}
	T^{1, 0}M \cap T^{0, 1}M = 0,
\end{equation}
where $T^{0, 1} M$ is the complex conjugate of $T^{1, 0} M$ in $T M \otimes \mathbb{C}$.
An important example of a three-dimensional CR manifold
is a real hypersurface $M$ in a two-dimensional complex manifold $X$;
this $M$ has the canonical CR structure
\begin{equation}
	T^{1, 0} M
	= T^{1, 0} X |_{M} \cap (T M \otimes \mathbb{C}).
\end{equation}
Introduce an operator $\overline{\partial}_{b} \colon C^{\infty}(M) \to \Gamma((T^{0, 1} M)^{*})$ by
\begin{equation}
	\overline{\partial}_{b} f = (d f)|_{T^{0, 1} M}.
\end{equation}
A smooth function $f$ is called a \emph{CR holomorphic function}
if $\overline{\partial}_{b} f = 0$.
A \emph{CR pluriharmonic function} is a real-valued smooth function
that is locally the real part of a CR holomorphic function.
We denote by $\mathcal{P}$ the space of CR pluriharmonic functions.

A CR structure $T^{1, 0} M$ is said to be \emph{strictly pseudoconvex}
if there exists a nowhere-vanishing real one-form $\theta$ on $M$
such that
$\theta$ annihilates $T^{1, 0} M$ and
\begin{equation}
	- \sqrt{- 1} d \theta (Z, \overline{Z}) > 0, \qquad
	0 \neq Z \in T^{1, 0} M;
\end{equation}
we call such a one-form a \emph{contact form}.
Denote by $T$ the \emph{Reeb vector field} with respect to $\theta$; 
that is, the unique vector field satisfying
\begin{equation}
	\theta(T) = 1, \qquad \iota_{T} d\theta = 0.
\end{equation}
Let $\tensor{Z}{_{1}}$ be a local frame of $T^{1, 0} M$,
and set $\tensor{Z}{_{\overline{1}}} = \overline{\tensor{Z}{_{1}}}$.
Then
$(T, \tensor{Z}{_{1}}, \tensor{Z}{_{\overline{1}}})$ gives a local frame of $T M \otimes \mathbb{C}$,
called an \emph{admissible frame}.
Its dual frame $(\theta, \tensor{\theta}{^{1}}, \tensor{\theta}{^{\overline{1}}})$
is called an \emph{admissible coframe}.
The two-form $d \theta$ is written as
\begin{equation}
	d \theta = \sqrt{- 1} \tensor{l}{_{1}_{\overline{1}}} \tensor{\theta}{^{1}} \wedge \tensor{\theta}{^{\overline{1}}},
\end{equation}
where $\tensor{l}{_{1}_{\overline{1}}}$ is a positive function.
We use $\tensor{l}{_{1}_{\overline{1}}}$ and its multiplicative inverse $\tensor{l}{^{1}^{\overline{1}}}$
to raise and lower indices.

A contact form $\theta$ induces a canonical connection $\nabla$,
called the \emph{Tanaka-Webster connection} with respect to $\theta$.
It is defined by
\begin{equation}
	\nabla T
	= 0,
	\quad
	\nabla \tensor{Z}{_{1}}
	= \tensor{\omega}{_{1}^{1}} \tensor{Z}{_{1}},
	\quad
	\nabla \tensor{Z}{_{\overline{1}}}
	= \tensor{\omega}{_{\overline{1}}^{\overline{1}}} \tensor{Z}{_{\overline{1}}}
	\quad
	\pqty{ \tensor{\omega}{_{\overline{1}}^{\overline{1}}}
	= \overline{\tensor{\omega}{_{1}^{1}}} }
\end{equation}
with the following structure equations:
\begin{gather}
	d \tensor{\theta}{^{1}}
	= \tensor{\theta}{^{1}} \wedge \tensor{\omega}{_{1}^{1}}
	+ \tensor{A}{^{1}_{\overline{1}}} \theta \wedge \tensor{\theta}{^{\overline{1}}}, \\
	d \tensor{l}{_{1}_{\overline{1}}}
	= \tensor{\omega}{_{1}^{1}} \tensor{l}{_{1}_{\overline{1}}}
	+ \tensor{l}{_{1}_{\overline{1}}} \tensor{\omega}{_{\overline{1}}^{\overline{1}}}.
\end{gather}
The tensor $\tensor{A}{_{1}_{1}} = \overline{\tensor{A}{_{\overline{1}}_{\overline{1}}}}$
is called the \emph{Tanaka-Webster torsion}.
We denote the components of a successive covariant derivative of a tensor
by subscripts preceded by a comma,
for example, $\tensor{K}{_{1}_{\overline{1}}_{,}_{1}}$;
we omit the comma if the derivatives are applied to a function.
We use the index $0$ for the component $T$ or $\theta$ in our index notation. 
The commutators of the second derivatives for $u \in C^{\infty}(M)$ are given by
\begin{equation}
\label{eq:commutator-of-covariant-derivatives}
	\tensor{u}{_{1}_{\overline{1}}} - \tensor{u}{_{\overline{1}}_{1}}
	= \sqrt{- 1} \tensor{l}{_{1}_{\overline{1}}} \tensor{u}{_{0}},
	\qquad
	\tensor{u}{_{0}_{1}} - \tensor{u}{_{1}_{0}}
	= \tensor{A}{_{1}_{1}} \tensor{u}{^{1}};
\end{equation}
see~\cite{Lee1988}*{(2.14)}.
Define the \emph{sub-Laplacian} $\Delta_{b}$ by
\begin{equation}
\label{eq:sub-Laplacian}
	\Delta_{b} u
	= - \tensor{u}{_{\overline{1}}^{\overline{1}}}
	-  \tensor{u}{_{1}^{1}}
\end{equation}
for $u \in C^{\infty}(M)$.
From \cref{eq:commutator-of-covariant-derivatives},
it follows that
\begin{equation}
\label{eq:another-form-of-sub-Laplacian}
	\Delta_{b} u
	= - 2 \tensor{u}{_{\overline{1}}^{\overline{1}}} - \sqrt{- 1} \tensor{u}{_{0}}
	= - 2 \tensor{u}{_{1}^{1}} + \sqrt{- 1} \tensor{u}{_{0}}.
\end{equation}
The curvature form
$\tensor{\Omega}{_{1}^{1}} = d \tensor{\omega}{_{1}^{1}}$
of the Tanaka-Webster connection satisfies
\begin{equation} \label{eq:curvature-form-of-TW-connection}
	\tensor{\Omega}{_{1}^{1}}
	= \Scal \cdot \tensor{l}{_{1}_{\overline{1}}} \tensor{\theta}{^{1}} \wedge \tensor{\theta}{^{\overline{1}}}
		- \tensor{A}{_{1}_{1}_{,}^{1}} \theta \wedge \tensor{\theta}{^{1}}
		+ \tensor{A}{_{\overline{1}}_{\overline{1}}_{,}^{\overline{1}}} \theta \wedge \tensor{\theta}{^{\overline{1}}},
\end{equation}
where $\Scal$ is the \emph{Tanaka-Webster scalar curvature}.

Let $\hat{\theta} = e^{\Upsilon} \theta$ be another contact form,
and denote by $\widehat{T}$ the corresponding Reeb vector field.
The admissible coframe corresponding to $(\widehat{T}, \tensor{Z}{_{1}}, \tensor{Z}{_{\overline{1}}})$
is given by
\begin{equation}
\label{eq:transformation-law-of-admissible-coframe}
	(\hat{\theta},
	\tensor{\hat{\theta}}{^{1}} = \tensor{\theta}{^{1}} + \sqrt{- 1} \tensor{\Upsilon}{^{1}} \theta,
	\tensor{\hat{\theta}}{^{\overline{1}}}
	= \tensor{\theta}{^{\overline{1}}} - \sqrt{- 1} \tensor{\Upsilon}{^{\overline{1}}} \theta).
\end{equation}
Under this conformal change,
the sub-Laplacian $\widehat{\Delta}_{b}$ with respect to $\hat{\theta}$ satisfies
\begin{equation}
\label{eq:transformation-law-of-sub-Laplacian}
	e^{\Upsilon} \widehat{\Delta}_{b} u
	= \Delta_{b} u - \tensor{\Upsilon}{^{1}} \tensor{u}{_{1}}
		- \tensor{\Upsilon}{^{\overline{1}}} \tensor{u}{_{\overline{1}}};
\end{equation}
see \cite{Stanton1989}*{Lemma 1.8} for example.

An important example of strictly pseudoconvex CR manifolds
is the boundary of a strictly pseudoconvex domain.
Let $\Omega$ be a relatively compact domain in a two-dimensional complex manifold $X$
with smooth boundary $M = \partial \Omega$.
Then there exists a smooth function $\rho$ on $X$ such that
\begin{equation}
	\Omega = \rho^{-1}((- \infty, 0)), \quad
	M = \rho^{- 1}(0), \quad
	d \rho \neq 0 \ \text{on} \ M;
\end{equation}
such a $\rho$ is called a \emph{defining function} of $\Omega$.
A domain $\Omega$ is said to be \emph{strictly pseudoconvex}
if we can take a defining function $\rho$ of $\Omega$
such that $d d^{c} \rho$ defines a K\"{a}hler form near $M$.
The boundary of a strictly pseudoconvex domain
is a closed strictly pseudoconvex real hypersurface;
the one-form $d^{c} \rho |_{M}$ is a contact form on $M$.

\section{CR pluriharmonic functions and pseudo-Einstein condition}
\label{section:CR-pluriharmonic-functions-and-pseudo-Einstein-condition}

Let $(M, T^{1, 0} M)$ be a strictly pseudoconvex CR manifold of dimension three
and $\theta$ a contact form on $M$.
We first discuss a CR analog of $d^{c}$.
Such an operator was first introduced by Rumin~\cite{Rumin1994} in his study of the Rumin complex.
Here,
however,
we define it as an operator taking values in $\Omega^{1}(M)$,
which implicitly appears in the proof of~\cite{Hirachi2014}*{Lemma 3.2}.

\begin{lemma}
	The differential operator
	\begin{equation}
	\label{eq:definition-of-d^c_CR}
		d^{c}_{\CR} \colon C^{\infty}(M) \to \Omega^{1} (M) ; \quad
		u \mapsto \frac{\sqrt{- 1}}{2} \pqty{\tensor{u}{_{\overline{1}}} \tensor{\theta}{^{\overline{1}}}
			- \tensor{u}{_{1}} \tensor{\theta}{^{1}}} + \frac{1}{2} \Delta_{b} u \cdot \theta
	\end{equation}
	is independent of the choice of $\theta$.
\end{lemma}

\begin{proof}
	Consider the conformal change $\hat{\theta} = e^{\Upsilon} \theta$.
	The independence of $d^{c}_{\CR}$ in $\theta$
	follows from the transformation laws 
	of an admissible coframe \cref{eq:transformation-law-of-admissible-coframe}
	and the sub-Laplacian \cref{eq:transformation-law-of-sub-Laplacian}.
\end{proof}

As in complex geometry,
CR pluriharmonic functions are smooth functions annihilated by $d d^{c}_{\CR}$.

\begin{lemma}
\label{lem:formula-of-dd^c_CR}
	For $u \in C^{\infty}(M)$,
	\begin{equation}
	\label{eq:formula-of-dd^c_CR}
		d d^{c}_{\CR} u
		= \tensor{P}{_{1}} u \cdot \theta \wedge \tensor{\theta}{^{1}}
			+ \tensor{P}{_{\overline{1}}} u \cdot \theta \wedge \tensor{\theta}{^{\overline{1}}},
	\end{equation}
	where
	\begin{equation}
		\tensor{P}{_{1}} u
		= \tensor{u}{_{\overline{1}}^{\overline{1}}_{1}}
			+ \sqrt{- 1} \tensor{A}{_{1}_{1}} \tensor{u}{^{1}},
		\qquad
		\tensor{P}{_{\overline{1}}} u
		= \tensor{u}{_{1}^{1}_{\overline{1}}}
			- \sqrt{- 1} \tensor{A}{_{\overline{1}}_{\overline{1}}} \tensor{u}{^{\overline{1}}}.
	\end{equation}
	In particular,
	u is a CR pluriharmonic function if and only if $d d^{c}_{\CR} u = 0$.
\end{lemma}

\begin{proof}
	We first show the equality \cref{eq:formula-of-dd^c_CR}.
	From \cref{eq:sub-Laplacian,eq:definition-of-d^c_CR},
	it follows that
	\begin{align}
		d d^{c}_{\CR} u
		&= \bqty{- \frac{1}{2} \tensor{(\Delta_{b} u)}{_{1}} - \frac{\sqrt{- 1}}{2} \tensor{u}{_{1}_{0}}
			+ \frac{\sqrt{- 1}}{2} \tensor{A}{_{1}_{1}} \tensor{u}{^{1}}} \theta \wedge \tensor{\theta}{^{1}} \\
		&\qquad + \bqty{- \frac{1}{2} \tensor{(\Delta_{b} u)}{_{\overline{1}}} + \frac{\sqrt{- 1}}{2} \tensor{u}{_{\overline{1}}_{0}}
			- \frac{\sqrt{- 1}}{2} \tensor{A}{_{\overline{1}}_{\overline{1}}} \tensor{u}{^{\overline{1}}}} \theta \wedge \tensor{\theta}{^{\overline{1}}}.
	\end{align}
	Hence it suffices to show
	\begin{equation}
		\tensor{P}{_{1}} u
		= - \frac{1}{2} \tensor{(\Delta_{b} u)}{_{1}} - \frac{\sqrt{- 1}}{2} \tensor{u}{_{1}_{0}}
			+ \frac{\sqrt{- 1}}{2} \tensor{A}{_{1}_{1}} \tensor{u}{^{1}};
	\end{equation}
	the other part is the complex conjugate of this equality.
	By using \cref{eq:another-form-of-sub-Laplacian,eq:commutator-of-covariant-derivatives},
	we have
	\begin{align}
		- \frac{1}{2} \tensor{(\Delta_{b} u)}{_{1}} - \frac{\sqrt{- 1}}{2} \tensor{u}{_{1}_{0}}
			+ \frac{\sqrt{- 1}}{2} \tensor{A}{_{1}_{1}} \tensor{u}{^{1}}
		&= \tensor{u}{_{\overline{1}}^{\overline{1}}_{1}} + \frac{\sqrt{- 1}}{2} \pqty{\tensor{u}{_{0}_{1}} - \tensor{u}{_{1}_{0}}}
			+ \frac{\sqrt{- 1}}{2} \tensor{A}{_{1}_{1}} \tensor{u}{^{1}} \\
		&= \tensor{u}{_{\overline{1}}^{\overline{1}}_{1}} + \sqrt{- 1} \tensor{A}{_{1}_{1}} \tensor{u}{^{1}} \\
		&= \tensor{P}{_{1}} u.
	\end{align}
	The latter statement is a consequence of~\cite{Lee1988}*{Proposition 3.4}.
\end{proof}

Each $u \in \mathcal{P}$ defines a cohomology class $[d^{c}_{\CR} u]$ in $H^{1}(M, \mathbb{R})$.
The following lemma is obtained from~\cite{Lee1988}*{Lemma 3.1}.

\begin{lemma}
\label{lem:real-part-of-CR-holomorphic}
	For $u \in \mathcal{P}$,
	the cohomology class $[d^{c}_{\CR} u]$ is equal to zero
	if and only if $u$ is the real part of a CR holomorphic function globally.
\end{lemma}

Assume that $M$ is realized as the boundary of a two-dimensional strictly pseudoconvex domain $\Omega$.
Take a defining function $\rho$ of $\Omega$ with $\theta = d^{c} \rho |_{M}$.

\begin{lemma}
\label{lem:characterization-of-d^c_CR-via-pluriharmonic-extension}
	For each $u \in C^{\infty}(M)$,
	there exists a smooth extension $\widetilde{u}$ of $u$ to $\overline{\Omega}$ such that
	$d d^{c} \widetilde{u} |_{M}$ has vanishing $(1, 1)$-part.
	Moreover,
	such a $\widetilde{u}$ is unique modulo $O(\rho^{2})$,
	and $d^{c} \widetilde{u} |_{M}$ coincides with $d^{c}_{\CR} u$.
\end{lemma}

\begin{proof}
	Take a smooth function $u^{\prime}$ on $\overline{\Omega}$ with $u^{\prime} |_{M} = u$.
	Then
	\begin{equation}
		d^{c} u^{\prime} |_{M}
		= \frac{\sqrt{- 1}}{2} (\tensor{u}{_{\overline{1}}} \tensor{\theta}{^{\overline{1}}}
			- \tensor{u}{_{1}} \tensor{\theta}{^{1}})
			+ \lambda \theta
	\end{equation}
	for some $\lambda \in C^{\infty}(M)$.
	Hence the $(1, 1)$-part of $d d^{c} u^{\prime} |_{M}$ is given by
	\begin{equation}
		\frac{\sqrt{- 1}}{2}
		\pqty{\tensor{u}{_{\overline{1}}_{1}} + \tensor{u}{_{1}_{\overline{1}}} + 2 \lambda \tensor{l}{_{1}_{\overline{1}}}}
		\tensor{\theta}{^{1}} \wedge \tensor{\theta}{^{\overline{1}}}.
	\end{equation}
	On the other hand,
	the $(1, 1)$-part of $d d^{c} (\rho v) |_{M}$ for $v \in C^{\infty} (\overline{\Omega})$ coincides with
	\begin{equation}
		\sqrt{- 1} v |_{M} \tensor{l}{_{1}_{\overline{1}}} \tensor{\theta}{^{1}} \wedge \tensor{\theta}{^{\overline{1}}}.
	\end{equation}
	If we choose $v$ so that $v |_{M} = 2^{- 1} \Delta_{b} u - \lambda$,
	the $(1, 1)$-part of $d d^{c} (u^{\prime} + \rho v) |_{M}$ vanishes,
	which gives the existence of $\widetilde{u}$.
	From the construction,
	it follows that $d^{c} \widetilde{u} |_{M} = d^{c}_{\CR} u$.
	The uniqueness of $\widetilde{u}$ modulo $O(\rho^{2})$
	is a consequence of the above computation of the $(1, 1)$-part of $d d^{c} (\rho v) |_{M}$.
\end{proof}

Set
\begin{equation}
	\tensor{W}{_{1}}
	= \tensor{\Scal}{_{1}} - \sqrt{- 1} \tensor{A}{_{1}_{1}_{,}^{1}},
	\qquad
	\tensor{W}{_{\overline{1}}}
	= \overline{\tensor{W}{_{1}}}
	= \tensor{\Scal}{_{\overline{1}}} + \sqrt{- 1} \tensor{A}{_{\overline{1}}_{\overline{1}}_{,}^{\overline{1}}},
\end{equation}
and define a two-form $W(\theta)$ on $M$ by
\begin{equation}
\label{eq:definition-of-representative}
	W(\theta)
	= \tensor{W}{_{1}} \theta \wedge \tensor{\theta}{^{1}}
		+ \tensor{W}{_{\overline{1}}} \theta \wedge \tensor{\theta}{^{\overline{1}}}.
\end{equation}
A contact form $\theta$ is said to be \emph{pseudo-Einstein} if $W(\theta) = 0$.

\begin{lemma}
	The two-form $W(\theta) / 2 \pi$ is closed
	and gives a representative of the first Chern class $c_{1}(T^{1, 0} M)$.
	In particular,
	$c_{1}(T^{1, 0} M) = 0$ in $H^{2}(M, \mathbb{R})$
	if $(M, T^{1, 0} M)$ admits a pseudo-Einstein contact form.
\end{lemma}

\begin{proof}
	It follows from \cref{eq:curvature-form-of-TW-connection} that
	\begin{equation}
		W(\theta)
		= \sqrt{- 1} \tensor{\Omega}{_{1}^{1}} -  d \pqty{\Scal \cdot \theta}
		= \sqrt{- 1} d \tensor{\omega}{_{1}^{1}} - d \pqty{\Scal \cdot \theta}.
	\end{equation}
	Hence $W(\theta) / 2 \pi$ is closed
	and represents the cohomology class $c_{1}(T^{1, 0} M)$.
\end{proof}

The two-form $W(\theta)$ depends on the choice of $\theta$,
but $d d^{c}_{\CR}$ appears in its transformation law under conformal change.

\begin{lemma}
	For another contact form $\hat{\theta} = e^{\Upsilon} \theta$,
	\begin{equation}
		W(\hat{\theta})
		= W(\theta) - 3 d d^{c}_{\CR} \Upsilon.
	\end{equation}
	In particular,
	in the case that $\theta$ is pseudo-Einstein,
	so is $\hat{\theta}$ if and only if $\Upsilon$ is CR pluriharmonic.
\end{lemma}

\begin{proof}
	Under the conformal change $\hat{\theta} = e^{\Upsilon} \theta$,
	\begin{equation}
		e^{\Upsilon} \tensor{\widehat{W}}{_{1}}
		= \tensor{W}{_{1}} - 3 \tensor{P}{_{1}} \Upsilon;
	\end{equation}
	see~\cite{Hirachi1993}*{Lemma 5.4} for example.
	Hence
	\begin{equation}
		W(\hat{\theta})
		= \tensor{\widehat{W}}{_{1}} \hat{\theta} \wedge \tensor{\hat{\theta}}{^{1}}
			+ \tensor{\widehat{W}}{_{\overline{1}}} \hat{\theta} \wedge \tensor{\hat{\theta}}{^{\overline{1}}}
		= W(\theta) - 3 d d^{c}_{\CR} \Upsilon,
	\end{equation}
	which establishes the formula.
\end{proof}

\section{CR Paneitz operator in dimension three}
\label{section:CR-Paneitz-operator}

Let $(M, T^{1, 0} M)$ be a closed three-dimensional strictly pseudoconvex CR manifold
and $\theta$ a contact form on $M$.
The \emph{CR Paneitz operator} $P_{\theta}$ is defined by
\begin{equation}
	P_{\theta} u
	= (\tensor{P}{_{1}} u) \tensor{}{_{,}^{1}}
	= (\tensor{P}{_{\overline{1}}} u) \tensor{}{_{,}^{\overline{1}}}.
\end{equation}
The CR Paneitz operator is formally self-adjoint
and annihilates CR pluriharmonic functions~\cite{Graham-Lee1988}*{Section 2}.
We first show an integral formula for $P_{\theta}$ by using $d^{c}_{\CR}$.

\begin{lemma}
\label{lem:cohomological-expression-of-CR-Paneitz}
	For $u \in C^{\infty}(M)$,
	\begin{equation}
		\int_{M} u (P_{\theta} u) \theta \wedge d \theta
		= - \int_{M} d^{c}_{\CR} u \wedge d d^{c}_{\CR} u.
	\end{equation}
\end{lemma}

\begin{proof}
	From the definition of $d^{c}_{\CR}$ \cref{eq:definition-of-d^c_CR} and \cref{lem:formula-of-dd^c_CR},
	it follows that
	\begin{align}
		\int_{M} d^{c}_{\CR} u \wedge d d^{c}_{\CR} u
		&= \frac{\sqrt{- 1}}{2}
			\int_{M} \bqty{ \tensor{u}{_{\overline{1}}} (\tensor{P}{_{1}} u)
					\tensor{\theta}{^{\overline{1}}} \wedge \theta \wedge \tensor{\theta}{^{1}}
					-  \tensor{u}{_{1}} (\tensor{P}{_{\overline{1}}} u)
					\tensor{\theta}{^{1}} \wedge \theta \wedge \tensor{\theta}{^{\overline{1}}}} \\
		&= \frac{1}{2} \int_{M} \bqty{\tensor{u}{^{1}} (\tensor{P}{_{1}} u)
			+ \tensor{u}{^{\overline{1}}} (\tensor{P}{_{\overline{1}}} u)} \theta \wedge d \theta \\
		&= - \int_{M} u (P_{\theta} u) \theta \wedge d \theta,
	\end{align}
	which completes the proof.
\end{proof}

Assume that $(M, T^{1, 0} M)$ is embeddable;
this means $(M, T^{1, 0} M)$ can be embedded in $\mathbb{C}^{N}$ for large $N$.
Without loss of generality,
we may assume that $M$ is connected.
Harvey and Lawson~\cite{Harvey-Lawson1975}*{Theorem 10.4} have shown that
$M$ bounds a two-dimensional strictly pseudoconvex Stein space.
Lempert~\cite{Lempert1995}*{Theorem 8.1} has proved that,
in this case,
$M$ can be realized as the boundary of a strictly pseudoconvex domain $\Omega$
in a two-dimensional complex projective manifold $X$.
Take a K\"{a}hler form $\omega$ on $X$
and a defining function $\rho$ of $\Omega$.
For sufficiently large $N > 0$,
the real $(1, 1)$-form
\begin{equation}
	\omega_{+}
	= N \omega - 2 d d^{c} \log (- \rho)
\end{equation}
defines an (even) asymptotically complex hyperbolic metric on $\Omega$;
see \cites{Epstein-Melrose-Mendoza1991,Guillarmou-Sa_Barreto2008,Matsumoto2016} for details.
Denote by $\Delta_{+}$ the (nonnegative) Riemannian Laplacian for $\omega_{+}$.

\begin{lemma}
\label{lem:harmonic-characterization-of-dd^c_CR}
	Let $F$ be a smooth function on $\overline{\Omega}$ with $F |_{M} = u$.
	The condition $\Delta_{+} F = O(\rho^{2})$ is equivalent to
	that $d d^{c} F |_{M}$ has vanishing $(1, 1)$-part.
	In particular,
	$d^{c} F |_{M} = d^{c}_{\CR} u$ if $\Delta_{+} F = O(\rho^{2})$.
\end{lemma}

\begin{proof}
	We first note that
	\begin{equation}
		- 4 d d^{c} F \wedge \omega_{+}
		= (\Delta_{+} F) \omega_{+}^{2}.
	\end{equation}
	It follows from the definition of $\omega_{+}$ that
	\begin{equation}
		\omega_{+}^{2}
		= - 8 \rho^{- 3} d \rho \wedge d^{c} \rho \wedge d d^{c} \rho + O(\rho^{- 2}).
	\end{equation}
	Take a local frame $\tensor{\widetilde{Z}}{_{1}}$ of $T^{1, 0} X \cap \Ker \partial \rho$,
	and set $\tensor{\widetilde{Z}}{_{\overline{1}}} = \overline{\tensor{\widetilde{Z}}{_{1}}}$.
	Since $M$ is strictly pseudoconvex,
	$d d^{c} \rho (\tensor{\widetilde{Z}}{_{1}}, \tensor{\widetilde{Z}}{_{\overline{1}}})$
	is non-zero near $M$,
	and
	\begin{equation}
		d d^{c} F \wedge \omega_{+}
		= 2 \rho^{- 2} \frac{d d^{c} F(\tensor{\widetilde{Z}}{_{1}}, \tensor{\widetilde{Z}}{_{\overline{1}}})}
			{d d^{c} \rho (\tensor{\widetilde{Z}}{_{1}}, \tensor{\widetilde{Z}}{_{\overline{1}}})}
			d \rho \wedge d^{c} \rho \wedge d d^{c} \rho + O(\rho^{- 1}).
	\end{equation}
	This proves the former statement.
	The latter one is a consequence of \cref{lem:characterization-of-d^c_CR-via-pluriharmonic-extension}.
\end{proof}

We need a solution to the Dirichlet problem for $\Delta_{+}$.
Near $M$,
the Laplacian $\Delta_{+}$ satisfies
\begin{equation}
\label{eq:Laplacian-of-ACH}
	\Delta_{+}
	= - \pqty{\rho \pdv{}{\rho}}^{2} + 2 \rho \pdv{}{\rho}
		+ (\text{a diff.\ op.\ increasing the order in $\rho$});
\end{equation}
c.f.\ \cite{Graham-Lee1988}*{Proposition 2.1}.

\begin{proposition}
\label{prop:Dirichlet-problem}
	For each $u \in C^{\infty}(M)$,
	there exist smooth functions $F$ and $G$ on $\overline{\Omega}$ such that
	$F |_{M} = u$ and $\widetilde{u} = F + G \rho^{2} \log (- \rho)$ satisfies $\Delta_{+} \widetilde{u} = 0$.
\end{proposition}

\begin{proof}
	For a smooth function $f$ on $\overline{\Omega}$,
	\cref{eq:Laplacian-of-ACH} implies that
	\begin{align}
		\Delta_{+}(f \rho^{k})
		&= k (2 - k) f \rho^{k} + O(\rho^{k + 1}), \\
		\Delta_{+}(f \rho^{k} \log (- \rho))
		&= k (2 - k) f \rho^{k} \log (- \rho) + 2 (1 - k) f \rho^{k} + O(\rho^{k + 1} \log (- \rho)).
	\end{align}
	By using an inductive argument and Borel's lemma,
	we obtain smooth functions $F_{0}$ and $G$ on $\overline{\Omega}$
	such that $F_{0} |_{M} = u$ and
	\begin{equation}
		H = \Delta_{+} (F_{0} + G \rho^{2} \log (- \rho)) \in \rho^{\infty} C^{\infty}(\overline{\Omega}),
	\end{equation}
	where $\rho^{\infty} C^{\infty}(\overline{\Omega})$ is the space of smooth functions on $\overline{\Omega}$
	that vanish to infinite order on $M$.
	Epstein, Melrose, and Mendoza~\cite{Epstein-Melrose-Mendoza1991} have proved
	that $\Delta_{+}$ has a bounded inverse $R$,
	and it maps $\rho^{\infty} C^{\infty}(\overline{\Omega})$ to $\rho^{2} C^{\infty}(\overline{\Omega})$;
	see also~\cite{Guillarmou-Sa_Barreto2008}*{Section 5.1}.
	If we set $F = F_{0} - R(H)$,
	then we have $F |_{M} = u$ and $\Delta_{+}(F + G \rho^{2} \log (- \rho)) = 0$.
\end{proof}

Note that $\Delta_{+} F = O(\rho^{2})$,
and so $d^{c} F |_{M} = d^{c}_{\CR} u$ by \cref{lem:harmonic-characterization-of-dd^c_CR}.
Now we give a proof of the nonnegativity of the CR Paneitz operator.

\begin{proof}[Proof of \cref{thm:nonnegativity-of-CR-Paneitz-operator}]
	Fix a smooth function $u$ on $M$,
	and let $\widetilde{u} = F + G \rho^{2} \log (- \rho)$ be as in \cref{prop:Dirichlet-problem}.
	Since $d d^{c} \widetilde{u} \wedge \omega_{+} = 0$,
	we have
	\begin{equation}
		d d^{c} \widetilde{u} \wedge d d^{c} \widetilde{u}
		\leq 0
	\end{equation}
	with equality if and only if $d d^{c} \widetilde{u} = 0$.
	The one-form $d^{c} (G \rho^{2} \log (- \rho))$ can be extended continuously to $\overline{\Omega}$
	and vanishes along $M$.
	On the other hand,
	\begin{equation}
	\label{eq:dd^c_CR-of-log-part}
		d d^{c} (G \rho^{2} \log (- \rho))
		\equiv \rho^{2} \log (- \rho) d d^{c} G
			+(2 \rho \log (- \rho) + \rho) d (G d^{c} \rho)
			\quad
			\text{(mod $d \rho$)}.
	\end{equation}
	The right hand side is continuous up to $M$
	and equal to zero on $M$.
	By using Stokes' theorem,
	we have
	\begin{align}
		\int_{M} d^{c}_{\CR} u \wedge d d^{c}_{\CR} u
		&= \int_{M} d^{c} F \wedge d d^{c} F \\
		&= \lim_{\epsilon \to + 0} \int_{\rho = - \epsilon} d^{c} \widetilde{u} \wedge d d^{c} \widetilde{u} \\
		&= \lim_{\epsilon \to + 0} \int_{\rho < - \epsilon} d d^{c} \widetilde{u} \wedge d d^{c} \widetilde{u} \\
		&\leq 0.
	\end{align}
	Hence \cref{lem:cohomological-expression-of-CR-Paneitz} yields
	the nonnegativity of the CR Paneitz operator.
	If the equality holds,
	then $\widetilde{u}$ must be pluriharmonic on $\Omega$,
	and so $d d^{c} \widetilde{u} \wedge d \rho = 0$ on $\Omega$.
	From \cref{eq:dd^c_CR-of-log-part},
	it follows that
	\begin{equation}
		d d^{c} (G \rho^{2} \log (- \rho)) \wedge d \rho
		= \bqty{\rho^{2} \log (- \rho) d d^{c} G
			+(2 \rho \log (- \rho) + \rho) d (G d^{c} \rho)} \wedge d \rho
	\end{equation}
	is continuous up to $M$ and vanishes along $M$.
	Therefore we have $d d^{c} F \wedge d \rho = 0$ on $M$,
	which implies $d d^{c} F |_{M} = d d^{c}_{\CR} u = 0$.
	In particular,
	$u$ is annihilated by $P_{\theta}$
	if and only if $u$ is CR pluriharmonic.
\end{proof}

From the above proof,
we obtain the following

\begin{corollary}
\label{cor:pluriharmonic-extension}
	Any CR pluriharmonic function $u$ on $M$
	admits a smooth pluriharmonic extension $\widetilde{u}$ to $\overline{\Omega}$,
	and $d^{c} \widetilde{u} |_{M} = d^{c}_{\CR} u$ holds.
\end{corollary}

\begin{proof}
	Let $u$ be a CR pluriharmonic function on $M$,
	and $\widetilde{u} = F + G \rho^{2} \log (- \rho)$ be as in \cref{prop:Dirichlet-problem}.
	From the proof of \cref{thm:nonnegativity-of-CR-Paneitz-operator},
	it follows that $\widetilde{u}$ is pluriharmonic on $\Omega$;
	in other words,
	$d d^{c} \widetilde{u} = 0$.
	On the other hand,
	$d d^{c} \widetilde{u}$ is equal to $2 G \log (- \rho) d \rho \wedge d^{c} \rho$
	modulo a two-form continuous up to the boundary.
	Thus we have $G |_{M} = 0$.
	The construction of $G$ implies that $G \in \rho^{\infty} C^{\infty}(\overline{\Omega})$;
	in particular,
	$\widetilde{u}$ is smooth up to the boundary.
	The equality $d^{c} \widetilde{u} |_{M} = d^{c}_{\CR} u$
	follows from $d d^{c} \widetilde{u} = 0$ and \cref{lem:characterization-of-d^c_CR-via-pluriharmonic-extension}.
\end{proof}

We apply \cref{cor:CR-Yamabe-constant},
which follows from \cref{thm:nonnegativity-of-CR-Paneitz-operator}
and~\cite{Cheng-Malchiodi-Yang2017}*{Theorem 1.2},
to give an affirmative solution to the CR Yamabe problem for embeddable CR manifolds.

\begin{proof}[Proof of \cref{thm:CR-Yamabe-problem}]
	If $\mathcal{Y}(M, T^{1, 0} M) < \mathcal{Y}(S^{3}, T^{1, 0} S^{3})$,
	then there exists a CR Yamabe contact form on $M$~\cite{Jerison-Lee1987}*{Theorem 3.4}.
	On the other hand,
	if $\mathcal{Y}(M, T^{1, 0} M) = \mathcal{Y}(S^{3}, T^{1, 0} S^{3})$,
	then $(M, T^{1, 0} M)$ is CR equivalent to $(S^{3}, T^{1, 0} S^{3})$ by \cref{cor:CR-Yamabe-constant},
	and a solution of the CR Yamabe problem exists~\cite{Jerison-Lee1987}*{Theorem 7.2}.
\end{proof}

\section{CR $Q$-curvature in dimension three}
\label{section:CR-Q-curvature}

In this section,
we discuss the zero CR $Q$-curvature problem in dimension three
and some applications.
Let $(M, T^{1, 0} M)$ be a closed three-dimensional strictly pseudoconvex CR manifold
and $\theta$ a contact form on $M$.
Hirachi~\cite{Hirachi1993} has introduced the \emph{CR $Q$-curvature} $Q_{\theta}$ by
\begin{equation}
	Q_{\theta}
	= \frac{1}{6}(\Delta_{b} \Scal - 2 \Im \tensor{A}{_{1}_{1}_{,}^{1}^{1}}).
\end{equation}
He has derived also the divergence formula
\begin{equation}
	Q_{\theta}
	= - \frac{1}{3} \tensor{W}{_{1}_{,}^{1}}
	= - \frac{1}{3} \tensor{W}{_{\overline{1}}_{,}^{\overline{1}}}
\end{equation}
and its transformation law under conformal change $\hat{\theta} = e^{\Upsilon} \theta$:
\begin{equation}
	e^{2 \Upsilon} Q_{\hat{\theta}}
	= Q_{\theta} + P_{\theta} \Upsilon;
\end{equation}
see~\cite{Hirachi1993}*{Lemma 5.4}.
In particular,
the CR $Q$-curvature vanishes for pseudo-Einstein contact forms.

We first discuss orthogonality relations between $Q_{\theta}$ and $\mathcal{P}$.
The integral of the product of a CR pluriharmonic function and the CR $Q$-curvature
has a cohomological expression.

\begin{proposition}
\label{prop:cohomological-interpretation-of-CR-Q-curvature}
	For $u \in \mathcal{P}$,
	\begin{equation} \label{eq:cohomological-interpretation-of-CR-Q-curvature}
		\langle [d^{c}_{\CR} u] \cup c_{1}(T^{1, 0} M), [M] \rangle
		= \frac{3}{2 \pi} \int_{M} u Q_{\theta} \theta \wedge d \theta.
	\end{equation}
\end{proposition}

\begin{proof}
	The definitions of $d^{c}_{\CR}$ \cref{eq:definition-of-d^c_CR} and $W(\theta)$ \cref{eq:definition-of-representative}
	yield that
	\begin{align}
		\int_{M} d^{c}_{\CR} u \wedge W(\theta)
		&= \frac{\sqrt{- 1}}{2} \int_{M}
			\bqty{\tensor{u}{_{\overline{1}}} \tensor{W}{_{1}}
			\tensor{\theta}{^{\overline{1}}} \wedge \theta \wedge \tensor{\theta}{^{1}}
			- \tensor{u}{_{1}} \tensor{W}{_{\overline{1}}}
			\tensor{\theta}{^{1}} \wedge \theta \wedge \tensor{\theta}{^{\overline{1}}}} \\
		&= \frac{1}{2} \int_{M} \bqty{\tensor{u}{^{1}} \tensor{W}{_{1}}
			+ \tensor{u}{^{\overline{1}}} \tensor{W}{_{\overline{1}}}} \theta \wedge d \theta \\
		&= 3 \int_{M} u Q_{\theta} \theta \wedge d \theta.
	\end{align}
	Since $W(\theta) / 2 \pi$ is a representative of $c_{1}(T^{1, 0} M)$,
	we have the desired conclusion.
\end{proof}

\cref{lem:real-part-of-CR-holomorphic} gives the following

\begin{corollary}
	If $u$ is the real part of a CR holomorphic function globally,
	then $Q_{\theta}$ is orthogonal to $u$.
\end{corollary}

Similarly,
the vanishing of the first Chern class implies the orthogonality of the CR $Q$-curvature to $\mathcal{P}$.

\begin{corollary}
\label{prop:orthogonality-with-vanishing-Chern}
	If $c_{1}(T^{1, 0} M) = 0$ in $H^{2}(M, \mathbb{R})$,
	then $Q_{\theta}$ is orthogonal to $\mathcal{P}$.
\end{corollary}

Now we give a complete solution to the zero CR $Q$-curvature problem
for embeddable CR manifolds.
To this end,
we recall some functional-analytic properties of the CR Paneitz operator obtained by Hsiao.

\begin{theorem}[\cite{Hsiao2015}*{Theorem 1.2}]
\label{thm:functional-analysis-of-CR-Paneitz}
	Let $(M, T^{1, 0} M)$ be a closed embeddable strictly pseudoconvex CR manifold of dimension three
	and $\theta$ a contact form on $M$.
	The maximal closed extension of $P_{\theta}$ is self-adjoint with closed range.
	Moreover,
	if $u \in \Dom P_{\theta} \cap (\Ker P_{\theta})^{\perp}$ satisfies $P_{\theta} u \in C^{\infty}(M)$,
	then $u$ must be smooth.
\end{theorem}

Denote by $\Pi_{\theta}$ the orthogonal projection to $\Ker P_{\theta}$.
We write $G_{\theta}$ for the partial inverse of $P_{\theta}$;
that is, $G_{\theta}$ is a bounded linear operator from $L^{2}(M)$ to $\Dom P_{\theta} \cap (\Ker P_{\theta})^{\perp}$ satisfying
\begin{gather}
	u = \Pi_{\theta} u + P_{\theta} G_{\theta} u \qquad u \in L^{2}(M), \\
	u = \Pi_{\theta} u + G_{\theta} P_{\theta} u \qquad u \in \Dom P_{\theta}.
\end{gather}
From the latter statement of \cref{thm:functional-analysis-of-CR-Paneitz},
it follows that $G_{\theta}$ maps $\Ran P_{\theta} \cap C^{\infty}(M)$ to itself,
and $\Pi_{\theta}$ defines a linear operator $C^{\infty}(M) \to \Ker P_{\theta} \cap C^{\infty}(M)$;
in particular,
$\mathcal{P} = \Ker P_{\theta} \cap C^{\infty}(M)$ is dense in $\Ker P_{\theta}$.

\begin{proof}[Proof of \cref{thm:zero-CR-Q-curvature}]
	Fix a contact form $\theta_{0}$ on $M$.
	If $Q_{\theta_{0}}$ is orthogonal to $\mathcal{P}$,
	then
	\begin{equation}
		\theta = \exp(- G_{\theta_{0}} Q_{\theta_{0}}) \cdot \theta_{0}
	\end{equation}
	is a contact form with zero CR $Q$-curvature.
	Hence it suffices to show
	\begin{equation}
		\langle [d^{c}_{\CR} u] \cup c_{1}(T^{1, 0} M), [M] \rangle
		= 0
	\end{equation}
	for any $u \in \mathcal{P}$.
	Take a smooth pluriharmonic extension $\widetilde{u}$ of $u$ to $\overline{\Omega}$
	by \cref{cor:pluriharmonic-extension}.
	Fix a Hermitian metric on $T^{1, 0} \overline{\Omega}$,
	and denote by $\Psi$ the corresponding first Chern form.
	By applying Stokes' theorem,
	we have
	\begin{equation}
		0
		= \int_{\Omega} d d^{c} \widetilde{u} \wedge \Psi
		= \int_{M} d^{c}_{\CR} u \wedge \Psi |_{M}
		= \langle [d^{c}_{\CR} u] \cup c_{1}(T^{1, 0} M), [M] \rangle;
	\end{equation}
	here we use the fact that $\Psi |_{M}$ is a representative
	of $c_{1}(T^{1, 0} M) = c_{1}(T^{1, 0} \overline{\Omega} |_{M})$.
	If $\hat{\theta} = e^{\Upsilon} \theta$ is also a contact form with zero CR $Q$-curvature,
	then $\Upsilon$ is annihilated by $P_{\theta}$,
	and so it is CR pluriharmonic by \cref{thm:nonnegativity-of-CR-Paneitz-operator}.
\end{proof}

Case and Yang~\cite{Case-Yang2013} have defined the \emph{$Q$-prime curvature} $Q^{\prime}_{\theta}$ by
\begin{equation}
	Q^{\prime}_{\theta}
	= \frac{1}{2} \Scal^{2} - 2 \abs{\tensor{A}{_{1}_{1}}}^{2} + \Delta_{b} \Scal.
\end{equation}
For any contact forms $\theta$ and $\hat{\theta} = e^{\Upsilon} \theta$,
the integral of the $Q$-prime curvature satisfies
\begin{equation}
\label{eq:transformation-law-of-total-Q-prime}
	\int_{M} Q^{\prime}_{\hat{\theta}} \hat{\theta} \wedge d \hat{\theta}
	= \int_{M} Q^{\prime}_{\theta} \theta \wedge d \theta
	+ 6 \int_{M} \bqty{\Upsilon (P_{\theta} \Upsilon) + 2 Q_{\theta} \Upsilon} \theta \wedge d \theta;
\end{equation}
see~\cite{Case-Yang2013}*{Proposition 6.1}.
In particular,
a discussion in~\cite{Case-Yang2013} yields that
the integral of the $Q$-prime curvature
\begin{equation}
	\overline{Q}^{\prime}(M, T^{1, 0} M)
	= \int_{M} Q^{\prime}_{\theta} \theta \wedge d \theta
\end{equation}
is independent of the choice of a contact form $\theta$ with zero CR $Q$-curvature,
and defines a CR invariant of $(M, T^{1, 0} M)$,
which we call the \emph{total $Q$-prime curvature}.
By \cref{thm:zero-CR-Q-curvature},
the total $Q$-prime curvature is always well-defined for closed embeddable CR three-manifolds.
Now we give a proof of \cref{thm:Einstein-condition},
which relates the total $Q$-prime curvature with the CR Yamabe constant.

\begin{proof}[Proof of \cref{thm:Einstein-condition}]
	Fix a contact form $\theta$ with zero CR $Q$-curvature.
	\cref{thm:CR-Yamabe-problem} implies that
	there exists $\Upsilon \in C^{\infty}(M)$ such that
	$\hat{\theta} = e^{\Upsilon} \theta$ satisfies
	\begin{equation}
		\widehat{\Scal}
		= \mathcal{Y}(M, T^{1, 0} M),
		\qquad
		\int_{M} \hat{\theta} \wedge d \hat{\theta}
		= 1.
	\end{equation}
	Then
	\begin{align}
		\int_{M} Q^{\prime}_{\hat{\theta}} \hat{\theta} \wedge d \hat{\theta}
		&= \frac{1}{2} \int_{M} (\widehat{\Scal})^{2} \hat{\theta} \wedge d \hat{\theta}
			- 2 \int_{M} \abs{\tensor{\widehat{A}}{_{1}_{1}}}^{2} \hat{\theta} \wedge d \hat{\theta} \\
		&\leq \frac{1}{2} \mathcal{Y}(M, T^{1, 0} M)^{2}
	\end{align}
	On the other hand,
	it follows from \cref{eq:transformation-law-of-total-Q-prime} and the nonnegativity of $P_{\theta}$ that
	\begin{align}
		\int_{M} Q^{\prime}_{\hat{\theta}} \hat{\theta} \wedge d \hat{\theta}
		&= \overline{Q}^{\prime}(M, T^{1, 0} M) + 6 \int_{M} \Upsilon (P_{\theta} \Upsilon) \theta \wedge d \theta \\
		&\geq \overline{Q}^{\prime}(M, T^{1, 0} M).
	\end{align}
	Therefore we get the desired inequality.
	Moreover,
	if the equality holds,
	then $\hat{\theta}$ has vanishing Tanaka-Webster torsion,
	and it is a pseudo-Einstein contact form since $\widehat{\Scal}$ is constant.
	Conversely,
	suppose that $(M, T^{1, 0} M)$ admits a pseudo-Einstein contact form $\theta$
	with vanishing Tanaka-Webster torsion;
	in this case, the Tanaka-Webster scalar curvature is constant.
	Without loss of generality,
	we may assume that $\int_{M} \theta \wedge d \theta = 1$.
	From \cref{thm:CR-Yamabe-problem} and \cite{Wang2015}*{Theorem 4},
	it follows that $\theta$ is a CR Yamabe contact form,
	and so the equality holds.
\end{proof}

\cref{thm:Einstein-condition} gives a generalization of~\cite{Case-Yang2013}*{Theorem 1.1}
to embeddable CR manifolds.

\begin{proof}[Proof of \cref{cor:characterization-of-sphere}]
	Since $0 \leq \mathcal{Y}(M, T^{1, 0} M) \leq \mathcal{Y}(S^{3}, T^{1, 0} S^{3})$,
	we have
	\begin{align}
		\mathcal{Y}(M, T^{1, 0} M)^{2}
		&\leq \mathcal{Y}(S^{3}, T^{1, 0} S^{3})^{2} \\
		&= 2 \overline{Q}^{\prime}(S^{3}, T^{1, 0} S^{3});
	\end{align}
	here the last equality follows from \cref{thm:Einstein-condition}
	and the fact that the standard contact form on $S^{3}$
	is a pseudo-Einstein contact form with vanishing Tanaka-Webster torsion.
	Moreover,
	if the equality holds,
	then $\mathcal{Y}(M, T^{1, 0} M) = \mathcal{Y}(S^{3}, T^{1, 0} S^{3})$,
	and $(M, T^{1, 0} M)$ is CR equivalent to $(S^{3}, T^{1, 0} S^{3})$ by \cref{cor:CR-Yamabe-constant}.
\end{proof}

\section{Logarithmic singularity of the Szeg\H{o} kernel}
\label{section:logarithmic-singularity-of-the-Szego-kernel}

Let $\Omega$ be a strictly pseudoconvex domain
in a two-dimensional complex manifold $X$.
We recall some facts on Fefferman defining functions;
see~\cite{Hislop-Perry-Tang2008}*{Section 2D} for example.
For a local coordinate $z = (z^{1}, z^{2})$ of $X$,
set
\begin{equation}
	\mathcal{J}_{z}[\varphi]
	= - \det
		\begin{pmatrix}
			\varphi & \pdvf{\varphi}{\overline{z}^{b}} \\
			\pdvf{\varphi}{z^{a}} & \partial^{2} \varphi / \partial z^{a} \partial \overline{z}^{b}
		\end{pmatrix}
\end{equation}
If $w = F(z)$ is another local coordinate,
then
\begin{equation}
	\mathcal{J}_{w}[\varphi]
	= \abs{\det F^{\prime}}^{2} \mathcal{J}_{z}[\varphi],
\end{equation}
where $F^{\prime}$ is the holomorphic Jacobi matrix of $F$.
A \emph{Fefferman defining function}
is a defining function $\rho$ of $\Omega$
such that
\begin{equation}
\label{eq:def-of-Fdf}
	d d^{c} \log \mathcal{J}_{z}[\rho]
	= d d^{c} O(\rho^{3})
\end{equation}
for any local coordinate $z$.
The condition \cref{eq:def-of-Fdf} is equivalent to 
that, for any $p \in \partial \Omega$,
there exists a local coordinate $z$ near $p$ such that
\begin{equation}
	\mathcal{J}_{z}[\rho]
	= 1 + O(\rho^{3}).
\end{equation}
It is known that $\theta$ is a pseudo-Einstein contact form on $\partial \Omega$
if and only if $\theta = d^{c} \rho |_{\partial \Omega}$ for a Fefferman defining function $\rho$ of $\Omega$;
see \cite{Hirachi1993}*{Lemma 7.2} and \cite{Hislop-Perry-Tang2008}*{Proposition 2.10} for example.
We say $\partial \Omega$ to be \emph{obstruction flat}
if we can take a defining function $\rho$ of $\Omega$ such that
$d d^{c} \log \mathcal{J}_{z}[\rho]$ vanishes to infinite order near a given point on $\partial \Omega$.

Now we give a necessary and sufficient condition
for the vanishing of $\psi_{\theta}$ in \cref{eq:aymptotic-expansion-of-Szego-kernel} to higher-order on the boundary.

\begin{proof}[Proof of \cref{thm:vanishing-of-logarithmic-singularity-of-Szego-kernel}]
	From \cref{eq:boundary-value-of-logarithmic-singularity},
	it follows that $\psi_{\theta} |_{\partial \Omega} = 0$ is equivalent to $Q_{\theta} = 0$.
	Since $\partial \Omega$ has a pseudo-Einstein contact form,
	$\psi_{\theta} |_{\partial \Omega} = 0$ if and only if $\theta$ is pseudo-Einstein
	by \cref{thm:zero-CR-Q-curvature}.
	In what follows,
	we suppose that $\theta$ is pseudo-Einstein.
	Take a Fefferman defining function $\rho$ with $\theta = d^{c} \rho |_{\partial \Omega}$.
	For a point $p \in \partial \Omega$,
	take a local coordinate $z$ near $p$ such that
	$\mathcal{J}_{z}[\rho] = 1 + O(\rho^{3})$.
	Hirachi, Komatsu, and Nakazawa~\cite{Hirachi-Komatsu-Nakawaza1993} have computed
	$\psi_{\theta}$ up to second order in this setting.
	From~\cite{Graham1987}*{Proposition 1.8}
	and~\cite{Hirachi-Komatsu-Nakawaza1993}*{Proposition 1 and Theorem 1},
	it follows that $\psi_{\theta} = O(\rho^{2})$ if and only if $\partial \Omega$ is obstruction flat.
	On the other hand,
	we derive from~\cite{Hirachi-Komatsu-Nakawaza1993}*{Remark 2}
	that $\psi_{\theta} = O(\rho^{3})$ if and only if $\partial \Omega$ is spherical.
\end{proof}

\section*{Acknowledgements}
The author is grateful to Kengo Hirachi for valuable suggestions
on the logarithmic singularity of the Szeg\H{o} kernel.
He would like to express his gratitude to Jeffrey Case, Sagun Chanillo, and Paul Yang for helpful comments.
He would like to thank the referees for their comments also,
which are helpful for the improvement of the manuscript.

\bibliography{my-reference,my-reference-preprint}

\end{document}